\documentclass[10pt]{article}

\usepackage{a4wide}
\usepackage{amssymb}
\usepackage{amsmath}
\DeclareMathOperator*{\argmin}{arg\,min}
\newcommand{\R}{\mathbb{R}}
\newcommand{\N}{\mathbb{N}}
\usepackage{amsthm}
\newtheorem{lemma}{Lemma}
\usepackage{gensymb}
\usepackage{graphicx}
\usepackage{bbold}
\usepackage{subfigure}
\usepackage[autostyle]{csquotes}
\usepackage{upgreek}
\usepackage{multirow}
\usepackage{hyperref}
\usepackage{xcolor}
\hypersetup{colorlinks, urlcolor = teal, linkcolor = teal, citecolor = teal}
\usepackage{algorithm}
\usepackage{algpseudocode}

\usepackage{authblk}

\usepackage{gensymb}

\begin{document}

\title{Learning Filter Functions in Regularisers\\by Minimising Quotients}

\author[1]{Martin Benning}
\author[2]{Guy Gilboa}
\author[1]{Joana Sarah Grah}
\author[1]{Carola-Bibiane Sch\"onlieb}

\affil[1]{University of Cambridge, Department of Applied Mathematics and Theoretical Physics\\

Centre for Mathematical Sciences, Wilberforce Road, Cambridge CB3 0WA, United Kingdom\\

\url{{mb941, jg704, cbs31}@cam.ac.uk}}

\affil[2]{Technion - Israel Institute of Technology, Electrical Engineering Department\\

Technion City, Haifa 32000, Israel\\

\url{guy.gilboa@ee.technion.ac.il}}

\date{March 16, 2017}

\maketitle

\begin{abstract}
Learning approaches have recently become very popular in the field of inverse problems. A large variety of methods has been established in recent years, ranging from bi-level learning to high-dimensional machine learning techniques. Most learning approaches, however, only aim at fitting parametrised models to favourable training data whilst ignoring misfit training data completely. In this paper, we follow up on the idea of learning parametrised regularisation functions by quotient minimisation as established in \cite{PAMM}. We extend the model therein to include higher-dimensional filter functions to be learned and allow for fit- and misfit-training data consisting of multiple functions. We first present results resembling behaviour of well-established derivative-based sparse regularisers like total variation or higher-order total variation in one-dimension. Our second and main contribution is the introduction of novel families of non-derivative-based regularisers. This is accomplished by learning favourable scales and geometric properties while at the same time avoiding unfavourable ones.
\end{abstract}

\section{Introduction}

Learning approaches for variational regularisation models constitute an active area of current research. In so-called bi-level learning approaches \cite{LearningRS,LearningKP}, for instance, one seeks to minimise a cost functional subject to a variational minimisation problem usually consisting of a data fidelity term and a regularisation term. Application of such models range from learning of suitable regularisation parameters to learning the correct operator or entire model, strongly dependent on the type of the underlying problem. In \cite{LearningRST}, the authors compared performance of Total Variation (TV), Infimal Convolution TV (ICTV) and second-order Total Generalised Variation (TGV$^2$) regularisers combined with both $L_1$ and $L_2$ cost functions for denoising of 200 images of the BSDS300 dataset measured by SSIM, PSNR and an objective value. There was no unique regulariser that always performed best. The images in the above-mentioned dataset differ from each other significantly enough such that advantages of the different regularisers become apparent for images with different prominent features such as sharp edges or piecewise linear regions. 

Another approach to variational regularisation learning is dictionary learning \cite{SparseRev}. 
In this approach the basic paradigm is that local image regions (patches) can be composed based on a linear combination of very few atoms from some dictionary. The dictionary could be global, for example wavelets or DCT-based, and in those cases a basis. However, it was revealed that tailored dictionaries to the specific image (or class of images), which are overcomplete, outperform the global dictionaries. These dictionaries are typically learned from the noisy image itself using algorithms such as K-SVD \cite{KSVD} based on orthogonal-matching-pursuit (OMP).
Recent studies have shown relations between convolutional neural nets and convolutional sparse coding \cite{SparseCoding}. 
One can conceptually perceive our proposed convolution filter set $\{h_i\}$, which defines the target-specific regulariser, as a small dictionary which is learned based on a few positive and negative image examples.

Learning approaches for variational regularisation models are aiming to design appropriate regularisation and customise it to particular structures present in the image. A somewhat separate route is the mathematical analysis of model-based regularisation, aiming at understanding the main building-blocks of existing regularisers to pave the path for designing new ones. In \cite{GroundStates}, for instance, the concept of ground states, singular values and singular vectors of regularisation functionals has been introduced, enabling the computation of solutions of variational regularisation schemes that can be reconstructed perfectly (up to a systematic bias). 

In \cite{PAMM} a new model motivated by generalised, non-linear Eigenproblems has been proposed to learn parametrised regularisation functions. As a novelty, both wanted and unwanted outcomes are incorporated in the model by integrating the former in the numerator and the latter in the denominator:
\begin{equation}
\hat{h} \in \argmin\limits_{\substack{\Vert h \Vert_2 = 1\\\text{mean}(h) = 0}} \frac{J(u^+;h)}{J(u^-;h)}, \quad J(u;h) = \Vert u \ast h \Vert_1,
\label{eq:basic}
\end{equation}
where $h$ is a parametrisation of a regularisation functional $J$ and $u^+$ and $u^-$ are desired and undesired input signals, respectively.

The basic idea underlying the model in \cite{PAMM} is optimisation of a quotient with respect to a convolution kernel parametrising certain regularisation functionals. In the paper, the authors investigate the same regularisation function, which is the one-norm of a signal convolved with a kernel $h$ both in the numerator and denominator. In the former, the input is a desirable signal, i.e.\ a function, which is preferred to be sparse once convolved with the kernel, whereas the latter is a signal to be avoided, yielding a large one-norm once convolved. In \cite{PAMM}, the undesirable signal has only been chosen to be pure noise or a noisy version of the desired input signal. In this work, however, we are also going to use clean signals as undesirable signals, with specific geometric properties or scales one wants to avoid. In Section \ref{sec:novfilters} we are going to see that this will enable us to derive tailored filters superior to those derived merely from desirable fitting data.

\section{The Proposed Learning Model}

In order to be able to incorporate multiple input functions, different regularisation functionals and multi-dimensional filter functions, we generalise the model in \cite{PAMM} as follows:
\begin{equation}
\hat{h} \in \argmin\limits_{\substack{\Vert h \Vert_2 = 1\\\text{mean}(h) = 0}} \frac{\frac{1}{M}\sum\limits_{i=1}^M\sum\limits_{k=1}^K J(u^+_i; h_k)}{\frac{1}{N}\sum\limits_{j=1}^N\sum\limits_{k=1}^K J(u^-_j; h_k)}\, , \quad J(u;h) = \Vert u \ast h \Vert_1.\label{eq:complex}
\end{equation}
Now, $\hat{h} = (\hat{h}_1, \dots, \hat{h}_K)$, where $\hat{h}_k \in \R^n$ for all $k \in \{1,\dots,K\}$, is a combination of multiple filter functions. The signals $u^+, u^- \in \R^m$ are one-dimensional or two-dimensional images written as a column vector. In the following section we want to describe how we want to solve \eqref{eq:complex} numerically.

\subsection{Numerical Implementation}
Viewing the quotients in \eqref{eq:basic} and \eqref{eq:complex} as generalised Rayleigh quotients, we observe that we deal with the (numerical) solution of generalised Eigenvalue problems. In order to solve \eqref{eq:basic} and \eqref{eq:complex} with the same algorithm, we write down an abstract algorithm for the solution of 
\begin{align} 
\hat{h} \in \argmin_{h} \left\{ \frac{F(h)}{G(h)} \quad \text{subject to} \quad \| h \|_2 = 1 \quad \text{and} \quad \text{mean}(h) = 0 \right\} \, .\label{eq:abstractprob}
\end{align}
The optimality condition of \eqref{eq:abstractprob} is given via $0 \in \partial F(\hat{h}) - \hat{\mu} \, \partial G(\hat{h})$, where $\partial F(\hat{h})$ and $\partial G(\hat{h})$ denote the subdifferential of $F$ and $G$ at $\hat{h}$, respectively, and $\hat{\mu} = F(\hat{h})/G(\hat{h})$. Note that the Lagrange multipliers for the constraints are zero by the same argumentation as in \cite[Section 2]{ISSD}, and can therefore be omitted.

In \cite{HeinBuehler} the authors have proposed a generalised inverse power method to tackle problems of the form \eqref{eq:abstractprob}. We, however, follow \cite{BLUB} and use a modification with added penalisation of the squared two-norm between $h$ and the previous iterate, to guarantee coercivity (and therefore existence and uniqueness of the solution) of the main update. The proposed algorithm for solving \eqref{eq:abstractprob} therefore reads as
\begin{align}
\begin{cases}
h^{k+\frac{1}{2}} &= \argmin\limits_{\text{mean}(h) = 0} \left\{ F(h) - \mu^k \langle h - h^k, s^k \rangle + \left\Vert h - h^k \right\Vert_2^2 \right\}\\
\mu^{k+1} &= \frac{F(h^{k + \frac{1}{2}})}{G(h^{k + \frac{1}{2}})}\\
s^{k+1} &\in \partial G(h^{k+\frac{1}{2}})\\
h^{k+1} &= \frac{h^{k+\frac{1}{2}}}{\left\Vert h^{k+\frac{1}{2}} \right\Vert_2}
\end{cases} \, .\label{eq:algorithm}
\end{align}

Similar to \cite{PAMM} we are using the CVX MATLAB\textsuperscript{\textregistered} software for disciplined convex programming \cite{CVX}. Due to the non-convexity of the overall problem and the resulting dependence on random initialisations of the filter, we re-initialise $h$ and iterate \eqref{eq:algorithm} 100 times. As also explained in \cite{PAMM}, reconstruction of a noisy signal in order to test the behaviour of the optimal filter is obtained by solving the following constrained optimisation problem:
\begin{equation}
\hat{u} = \argmin\limits_{u \in \mathbb{R}^m} J(u; \hat{h}) \qquad \text{subject to} \quad \Vert u - f \Vert_2 \leq \eta \sigma \sqrt{m}\,,
\label{eq:reconstruction}
\end{equation}
where $f$ is the sum of $u^+$ and Gaussian noise with zero mean and variance $\sigma^2$, $\eta$ is a weighting factor and $m$ is the number of elements of $u^+$.
\paragraph{Remark.}
In the setting of \cite{PAMM}, there is indeed an even more efficient way of finding suitable filter functions $h$. Simplifying the model to a variant without the need of having a negative input function $u^-$ yields the same results, which is a clear indicator that in the above-mentioned framework the numerator plays a dominant role. In fact, varying model \eqref{eq:basic} by replacing the one-norm in the denominator by $\Vert h \Vert_2$ returns exactly the same solutions. However, we would like to stress that the denominator is going to play a more important role in our extended model, since we are able to incorporate more than one input function $u^-$, especially ones which are different from pure noise. In fact, one can think of a large variety of undesired input signals such as specific textures and shapes.\\

Despite existing convergence results previously stated in \cite{HeinBuehler} and \cite{BLUB} we want to briefly state a simplified convergence result for global convergence of Algorithm \ref{eq:algorithm} in the following.

\subsection{A brief convergence analysis}
Following \cite[Section 3.2]{PALM}, we show two results that are essential for proving global convergence of Algorithm \eqref{eq:algorithm}: a descent lemma and a bound of the subgradient by the iterates gap. We start with the sufficient decrease property of the objective.
\begin{lemma}\label{lem:suffdecrease}
Let $F$ and $G$ be proper, lower semi-continuous and convex functions. Then the iterates of Algorithm \eqref{eq:algorithm} satisfy
\begin{align*}
\mu^{k + \frac{1}{2}} + \frac{1}{G(h^{k + \frac{1}{2}})}\| h^{k + \frac{1}{2}} - h^k \|^2 \leq \mu^k \, ,
\end{align*}
if we further assume $G(h^{k + \frac{1}{2}}) \neq 0$ for all $k \in \N$.
\end{lemma}
\begin{proof}
From the first equation of Algorithm \eqref{eq:algorithm} we observe
\begin{align*}
F(h^{k + \frac{1}{2}}) + \| h^{k + \frac{1}{2}} - h^k \|_2^2 &\leq F(h^k) + \mu^k \langle s^k, h^{k + \frac{1}{2}} - h^k \rangle \\
&\leq F(h^k) + \mu^k \left( G(h^{k + \frac{1}{2}}) - G(h^k) \right) \\
&= \mu^k G(h^{k + \frac{1}{2}}) \, ,
\end{align*}
due to the convexity of $G$. If we divide by $G(h^{k + \frac{1}{2}})$ on both sides of the equation, we obtain
\begin{align*}
\mu^{k + \frac{1}{2}} + \frac{1}{G(h^{k + \frac{1}{2}})}\| h^{k + \frac{1}{2}} - h^k \|^2 \leq \mu^k \, ,
\end{align*}
which concludes the proof.
\end{proof}
In order to further prove a bound of the subgradient by the iterates gap, we assume that $G$ is smooth and further has a Lipschitz-continuous gradient $\nabla G$. We want to point out that this excludes choices for $G$ such as in \eqref{eq:basic} and \eqref{eq:complex}, as the one-norm is neither smooth nor are its subgradients Lipschitz-continuous. A remedy here is the smoothing of the one-norms in \eqref{eq:basic} and \eqref{eq:complex}. If we replace the one-norm(s) in the denominator with Huber one-norms, i.e. we replace the modulus in the one-norm with the Huber function
\begin{align*}
\phi_\gamma(x) =
\begin{cases}
\frac{x^2}{2}, & \vert x \vert \leq \gamma\\
\gamma \left( \vert x \vert - \frac{\gamma}{2} \right), & \vert x \vert > \gamma
\end{cases} \, ,
\end{align*}
we can achieve smoothness and Lipschitz-continuity of the gradient, where the Lipschitz parameter depends on the smoothing parameter $\gamma$. We want to note that for $\gamma$ small enough we have not seen any significant difference in numerical performance between using the one-norm or its Huber counterpart.
\begin{lemma}\label{lem:subgradbound}
Let $F$ and $G$ be proper, lower semi-continuous and convex functions, and let $G$ be differentiable with $L$-Lipschitz-continuous gradient, i.e. $\| \nabla G(h_1) - \nabla G(h_2) \|_2 \leq L \| h_1 - h_2 \|_2$ for all $h_1$ and $h_2$ and a fixed constant $L$. Then the iterates of Algorithm \eqref{eq:algorithm} satisfy
\begin{align*}
\| r^{k + \frac{1}{2}} - \mu^{k + \frac{1}{2}} \nabla G(x^{k + \frac{1}{2}}) \|_2 \leq (2 + C^{k + \frac{1}{2}} L) \| h^{k + \frac{1}{2}} - h^k \|_2 \, ,
\end{align*}
for some constant $C^{k + \frac{1}{2}}$, $r^{k + \frac{1}{2}} \in \partial F(h^{k + \frac{1}{2}})$ and $\mu^{k + \frac{1}{2}} := \mu^{k + 1} = F(h^{k + \frac{1}{2}})/G(h^{k + \frac{1}{2}})$.
\end{lemma}
\begin{proof}
This follows almost instantly from the optimality condition and the Lipschitz-continuity of $\nabla G$. We obtain 
\begin{align}
r^{k + \frac{1}{2}} - \mu^k \nabla G(h^k) = 2(h^k - h^{k + \frac{1}{2}}) \, ,\label{eq:optcond}
\end{align}
for $r^{k + \frac{1}{2}} \in \partial F(h^{k + \frac{1}{2}})$, as the optimality condition of the first sub-problem of \eqref{eq:algorithm} - note that we can omit the zero-mean constraint with a similar argumentation as earlier. Hence, we obtain
\begin{align*}
\| r^{k + \frac{1}{2}} - \mu^{k + \frac{1}{2}} \nabla G(x^{k + \frac{1}{2}}) \|_2 &= \| r^{k + \frac{1}{2}} - \mu^k \nabla G(x^k) + \mu^k \nabla G(x^k) - \mu^{k + \frac{1}{2}} \nabla G(x^{k + \frac{1}{2}}) \|_2\\
&= \| 2(h^k - h^{k + \frac{1}{2}}) + \mu^k \nabla G(x^k) - \mu^{k + \frac{1}{2}} \nabla G(x^{k + \frac{1}{2}}) \|_2\\
&\leq 2 \| h^{k + \frac{1}{2}} - h^k \|_2 + C^{k + \frac{1}{2}} \| \nabla G(x^{k + \frac{1}{2}}) - \nabla G(x^k) \|_2
\intertext{thanks to \eqref{eq:optcond} and the triangle inequality. The constant $C^{k + \frac{1}{2}}$ equals either $\mu^k$ or $\mu^{k + \frac{1}{2}}$, depending on whether $\| \mu^{k + \frac{1}{2}} \nabla G(x^{k + \frac{1}{2}}) - \mu^k \nabla G(x^k) \|_2 \leq \mu^{k + \frac{1}{2}} \| \nabla G(x^{k + \frac{1}{2}}) - \nabla G(x^k) \|_2$ or $\| \mu^{k + \frac{1}{2}} \nabla G(x^{k + \frac{1}{2}}) - \mu^k \nabla G(x^k) \|_2 \leq \mu^k \| \nabla G(x^{k + \frac{1}{2}}) - \nabla G(x^k) \|_2$. Using the Lipschitz-continuity of $G$ then yields}
\| r^{k + \frac{1}{2}} - \mu^{k + \frac{1}{2}} \nabla G(x^{k + \frac{1}{2}}) \|_2 &\leq 2 \| h^{k + \frac{1}{2}} - h^k \|_2 + C^{k + \frac{1}{2}} L \| h^{k + \frac{1}{2}} - h^k \|_2\\
&= (2 + C^{k + \frac{1}{2}}L) \| h^{k + \frac{1}{2}} - h^k \|_2 \, .
\end{align*}
This concludes the proof.
\end{proof}
Under the additional assumption that the function $F/G$ satisfies the Kurdyka-\L ojasiewicz property (cf. \cite{lojasiewicz1963,kurdyka1998gradients}) we can now use Lemma \ref{lem:suffdecrease} and Lemma \ref{lem:subgradbound} to show finite length of the iterates \eqref{eq:algorithm} similar to \cite[Theorem 1]{PALM}, following the general recipe of \cite[Section 3.2]{PALM}. Note that we further have to substitute $h^{k + \frac{1}{2}} = \| h^{k + \frac{1}{2}} \|_2 h^{k + 1}$ in order to also show global convergence of the normalised iterates.

\section{Reproducing Standard Sparse Penalties}

In this section we want to demonstrate that we are able to reproduce standard first- and second-order total variation regularisation penalties in 1D.

\begin{figure}[h]
\centering
\includegraphics[height=6cm]{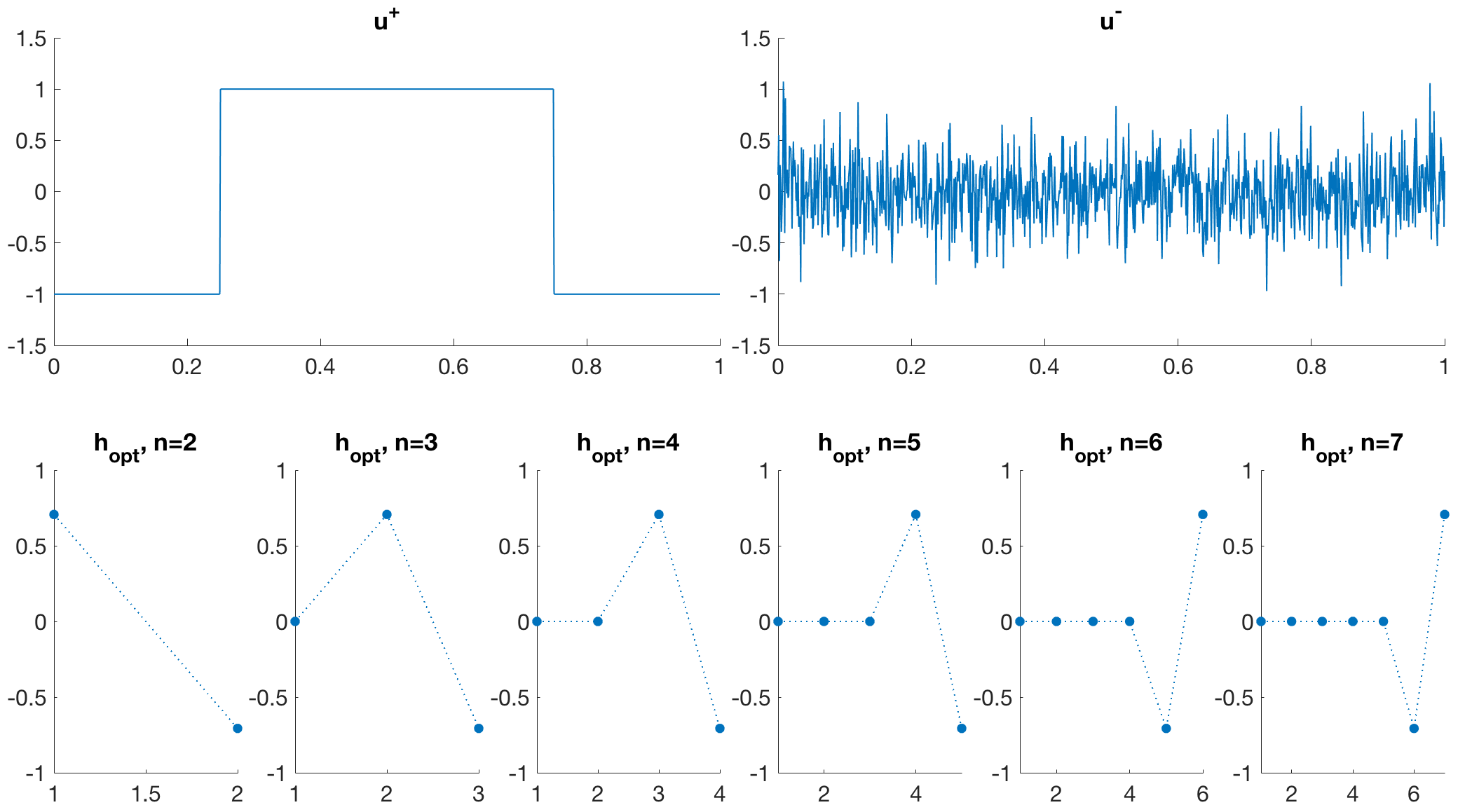}
\caption{Optimal filters in 1D setting for different sizes of $h$. Top: $u^+$ (left) and $u^-$ (right). Bottom: Optimal filters choosing $n \in \{ 2, 3,\dots, 7 \}$ (from left to right).}
\label{fig:differentsizes1D}
\end{figure}

Figure \ref{fig:differentsizes1D} shows results for different sizes of the kernel $h$. In all experiments the filter function is indeed resembling a two-point stencil functioning as a finite differences discretisation of TV. This is expected as the desired input function is a TV Eigenfunction.

\begin{figure}[h]
\centering
\subfigure[Top: $u^+_1$, $u^+_2$, $u^+_3$. Bottom: $u^-_1$, $u^-_2$, $\hat{h}$ ($n=5$).]{\includegraphics[height=4.75cm]{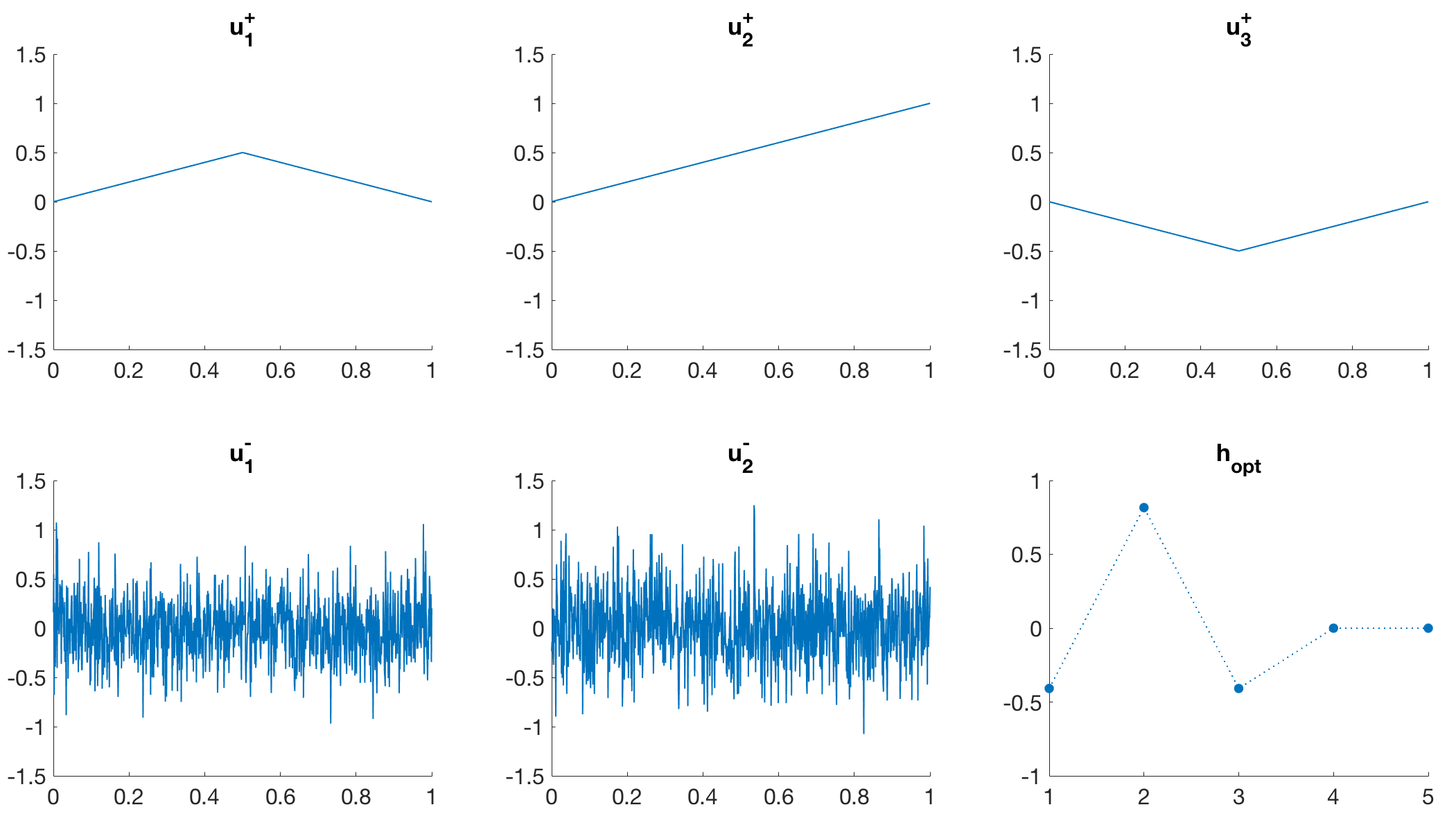}}\hfill\subfigure[Reconstruction with $\eta = 3.5$.]{\includegraphics[height=4.75cm]{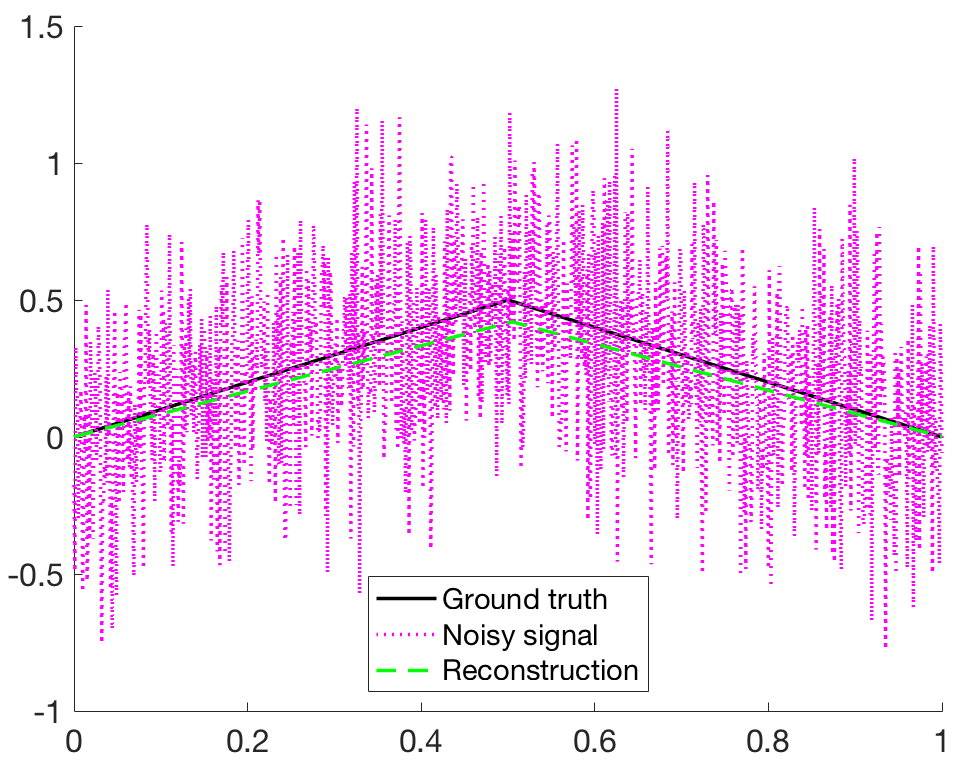}}
\caption{1D result for multiple piecewise-linear input functions $u^+_i$ and noisy signals $u^-_j$.}
\label{fig:1Dsecond}
\end{figure}

In Figure \ref{fig:1Dsecond} (a) we can reproduce a filter resembling a second-order derivative. This is indeed expected as we choose three different piecewise-linear functions as desired input signals. The reconstruction in (b) is performed according to \eqref{eq:reconstruction}.

In a more sophisticated example, we mimic a TV-TV$^2$ infimal convolution model, where we are given a known decomposition $u^+ = v+w$, i.e.\ $u^+$ consists of a smooth part $v$ and a piecewise constant part $w$. When minimising
\begin{equation*}
\frac{\Vert h_1 \ast v \Vert_1 + \Vert h_2 \ast w \Vert_1}{\Vert h_1 \ast u^- \Vert_1 + \Vert h_2 \ast u^- \Vert_1}\,,
\end{equation*}
with respect to $h_1$ and $h_2$, we indeed obtain two filters resembling a second- and first-order derivative, respectively (cf.\ Figure \ref{fig:1Dinfconv}).

\begin{figure}[h]
\centering
\includegraphics[height=3cm]{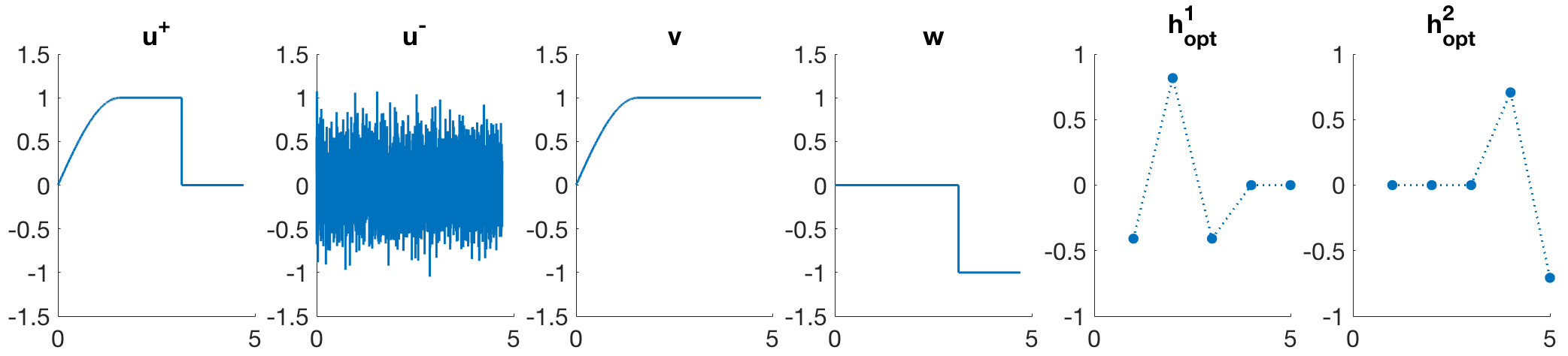}
\caption{1D infimal convolution result for two filters ($n=5$) and known decomposition $u^+ = v + w$. From left to right: $u^+$, $u^-$, $v$, $w$, $\hat{h}_1$, $\hat{h}_2$.}
\label{fig:1Dinfconv}
\end{figure}

\section{Novel Sparse Filters}\label{sec:novfilters}

In this section we derive a new family of regularisers not necessarily related to derivatives in contrast to the total variation. They have the interesting property of reconstructing piecewise-constant both vertical and horizontal lines in the corresponding null-spaces. Consequently, we are able to almost perfectly reconstruct those types of images and obtain better denoising results compared to standard TV denoising. In \cite{G}, a definition of desirable features of a regulariser, which is adapted for a specific type of images, is given. It is suggested that in the ideal case, all instances belonging to the desired clean class should be in the null-space of the regulariser (see \cite[Section 2]{G} and also compare \cite{BKC}). This is exactly what we obtain in the following.

\begin{figure}[h]
\centering
\subfigure[Optimal filter and reconstruction for thick vertical stripes.]{\includegraphics[height=4.25cm]{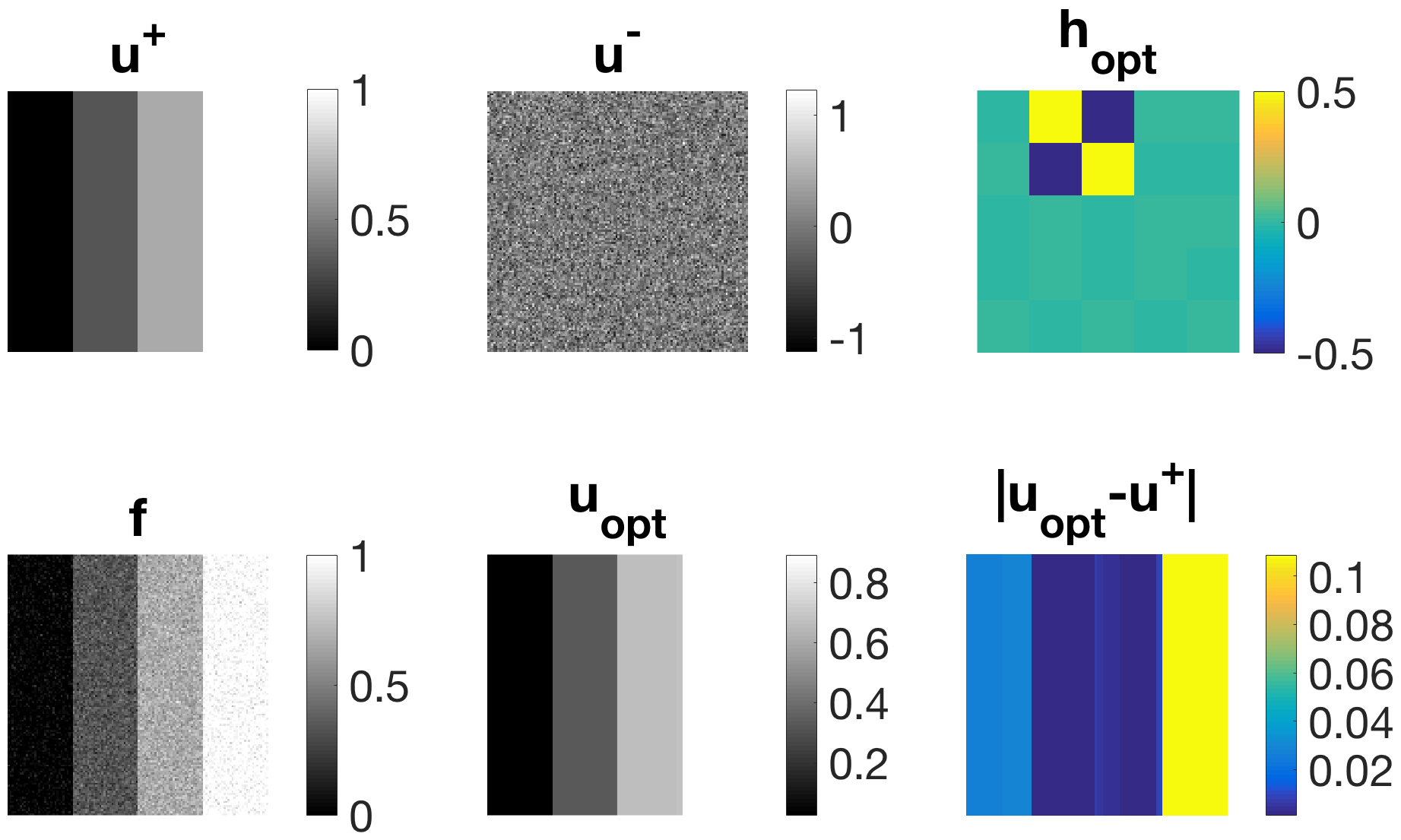}}\hfill\subfigure[Optimal filter and reconstruction for thin horizontal stripes.]{\includegraphics[height=4.25cm]{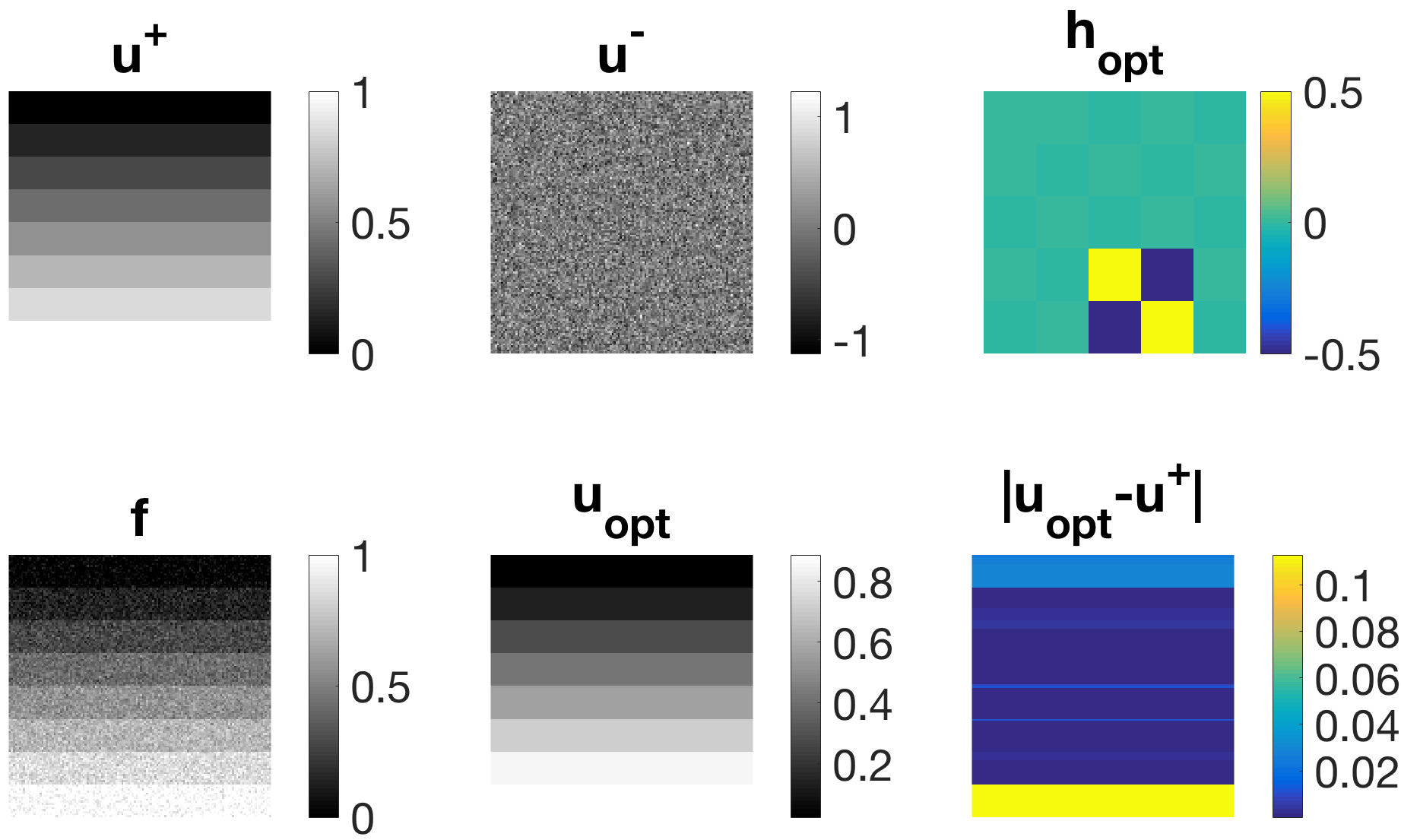}}
\caption{Experiments for 2D piecewise-constant images. Top: $u^+$ (ground truth for reconstruction), $u^-$ (pure Gaussian noise, $\sigma = 0.3$), $\hat{h}$. Bottom: Noisy image $f$ (Gaussian noise added to $u^+$, $\sigma^2 = 0.005$), reconstruction $\hat{u}$ ($\eta = 1$), absolute difference between reconstruction and ground truth.}
\label{fig:2Dpc_diagonal}
\end{figure}

In Figure \ref{fig:2Dpc_diagonal}, a new family of diagonal regularisers is established for piecewise-constant images with stripes in both vertical and horizontal direction. For denoising purposes, those filters yield superior results over TV denoising as they additionally avoid loss of contrast, which would occur when performing TV denoising for these examples.
The reason for that is simply that if we consider a $2 \times 2$ diagonal-shaped filter $h = [1,-1;-1,1]$ in the variational problem \eqref{eq:reconstruction}, we can expect both horizontal and vertical stripes to be in its null-space. Therefore, we obtain perfect shape preservation for any regularisation parameter $\eta$. Note that 1-pixel-thick stripes are in the null-space as well (cf.\ Figure \ref{fig:advantages} (left), where $f - u = 0$).

\begin{figure}[h]
\centering
\includegraphics[height=3.75cm]{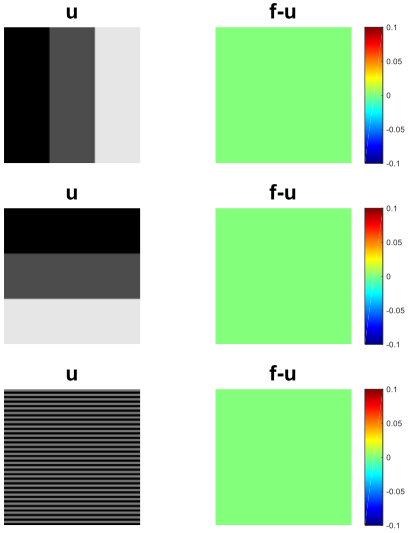}\hfill\includegraphics[height=3.75cm]{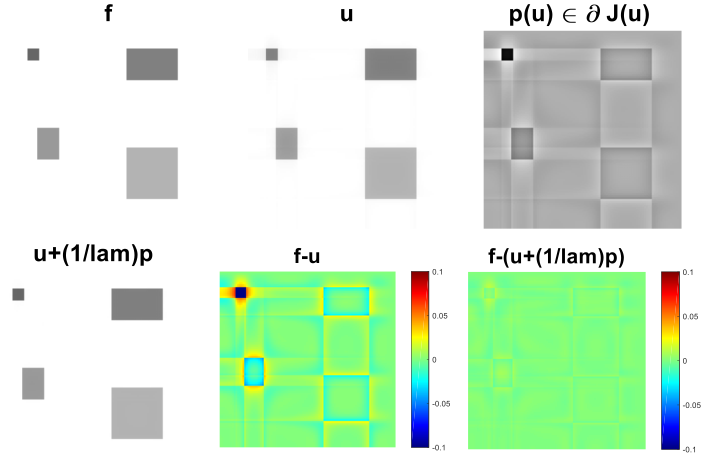}\hfill\includegraphics[height=3.75cm]{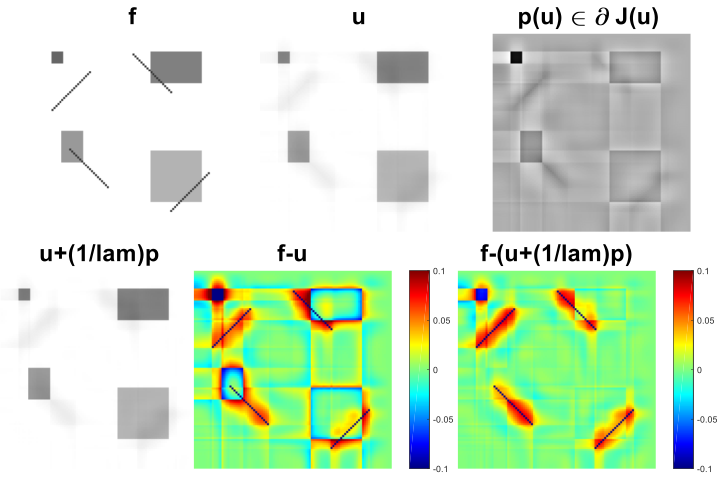}
\caption{The stripe images are in the novel diagonal filter's null-space (left). Rectangles can be well preserved (centre) and thin diagonal structures can be removed (right).}
\label{fig:advantages}
\end{figure}

In Figure \ref{fig:advantages} (centre) we can observe that rectangles are also well preserved with this filter. For better performance, however, one could additionally use a contrast preserving mechanism such as Bregman iteration or as in our case low-pass filtering. On the right, Figure \ref{fig:advantages} illustrates how the diagonal filter is capable of removing diagonal structures in an image.

However, for denoising tasks this filter is not optimal. For images that do not consist exclusively of horizontal or vertical stripes it produces undesired artefacts in the reconstruction such as additional thin stripes in horizontal and vertical direction. Also, we obtain the same filter despite having used quite different training data (rotated by 90\degree) in both experiments.

As a consequence, we are going to focus in the following on distinguishing between different shapes or scales.

\begin{figure}
\centering
\subfigure[$J(u^+;\hat{h})=286$, $J(u^-;\hat{h})=586$]{\includegraphics[height=4.25cm]{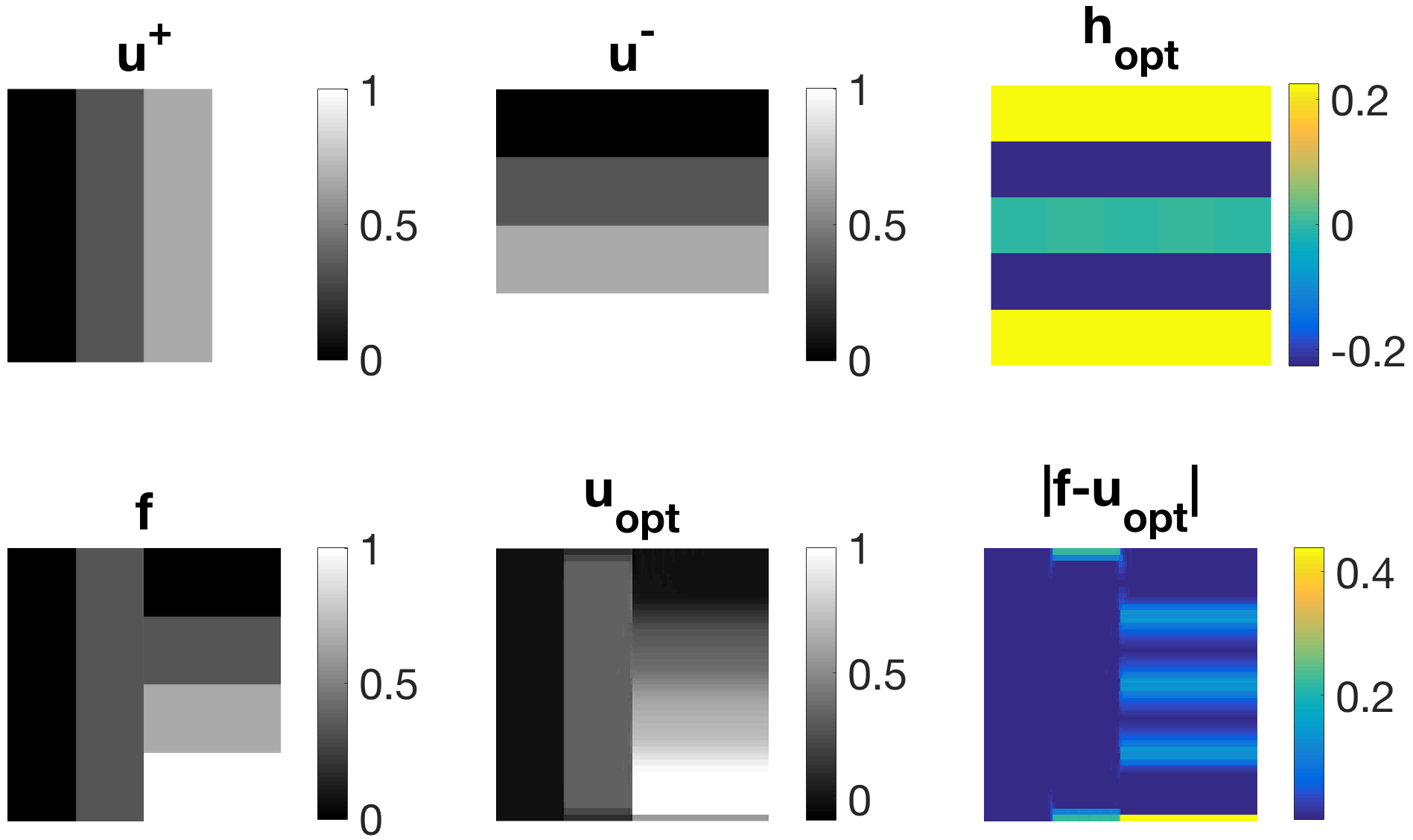}}\hfill
\subfigure[$J(u^+;\hat{h})=87$, $J(u^-;\hat{h})=186$]{\includegraphics[height=4.25cm]{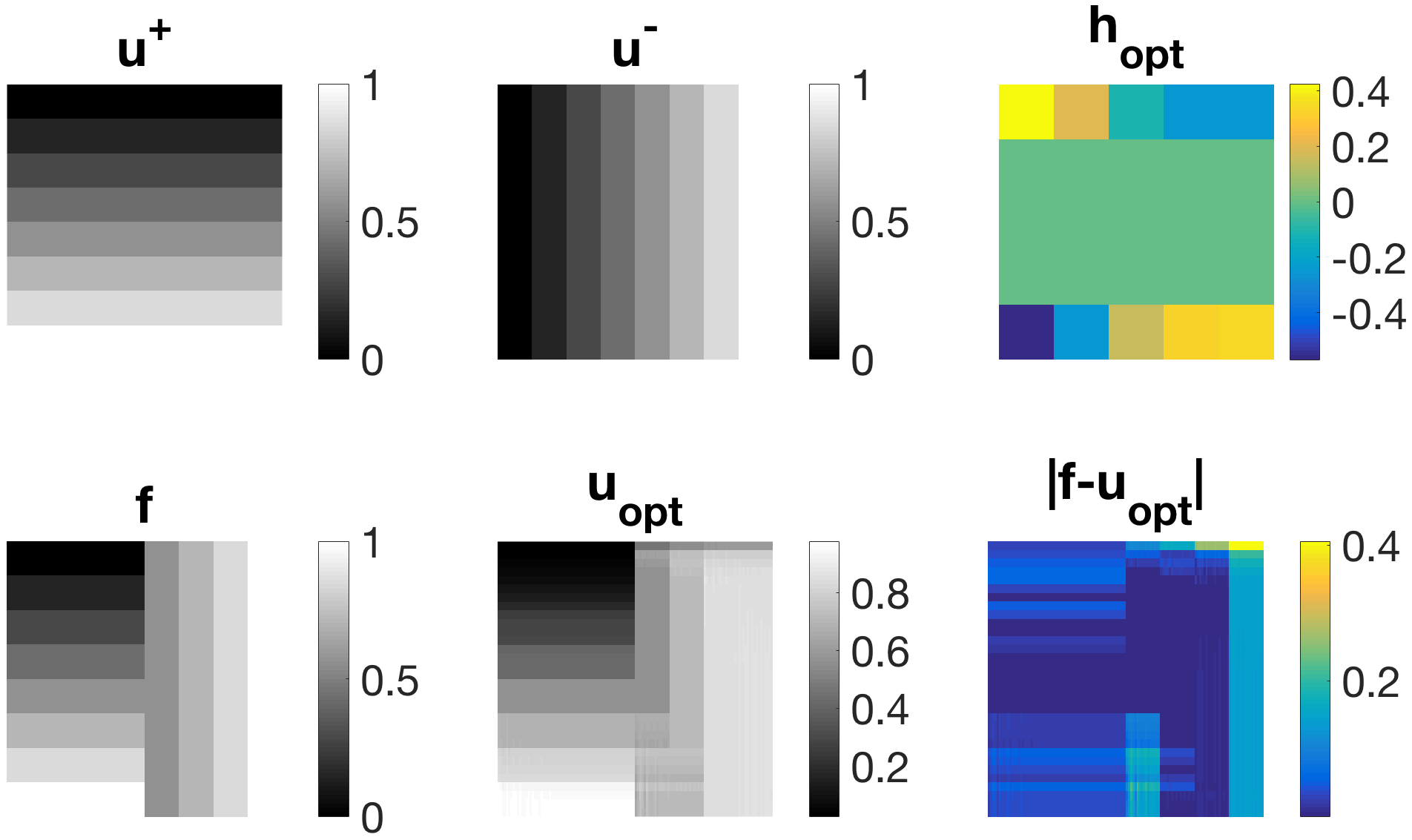}}\\
\subfigure[$J(u^+;\hat{h})=23$, $J(u^-;\hat{h})=1474$]{\includegraphics[height=4.25cm]{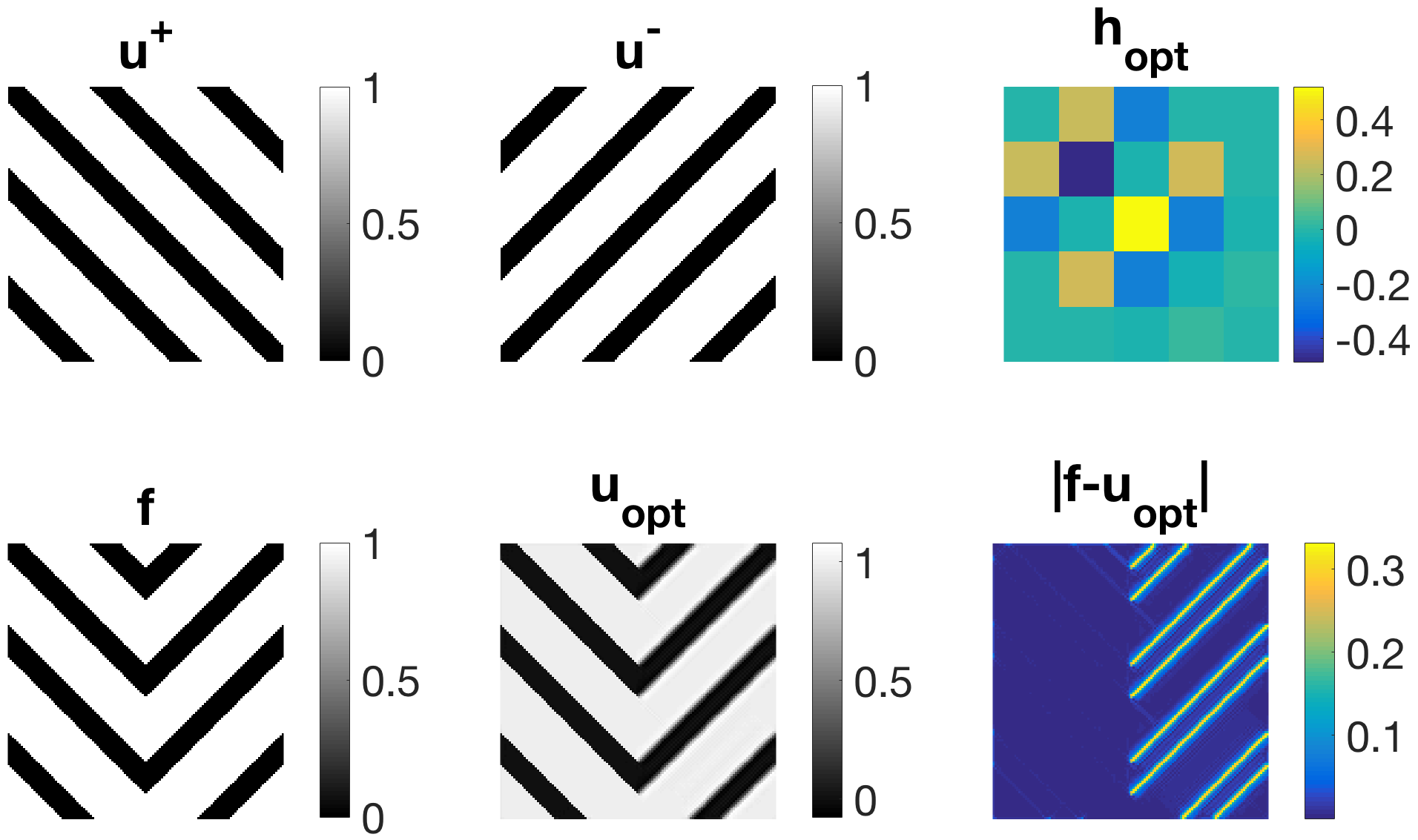}}\hfill
\subfigure[$J(u^+;\hat{h})=178$, $J(u^-;\hat{h})=11220$]{\includegraphics[height=4.25cm]{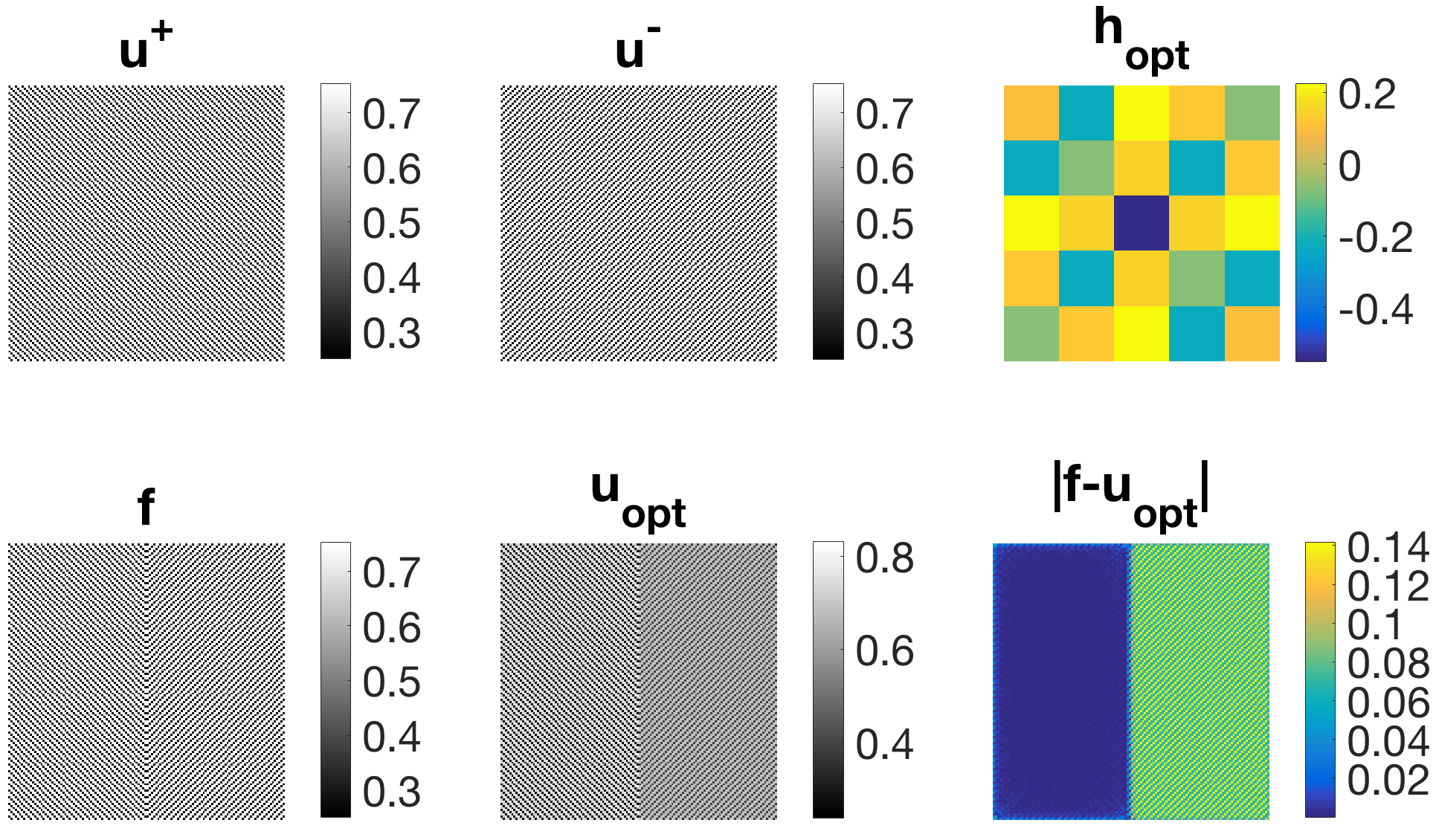}}\\
\subfigure[$J(u^+;\hat{h})=129$, $J(u^-;\hat{h})=277$]{\includegraphics[height=4cm]{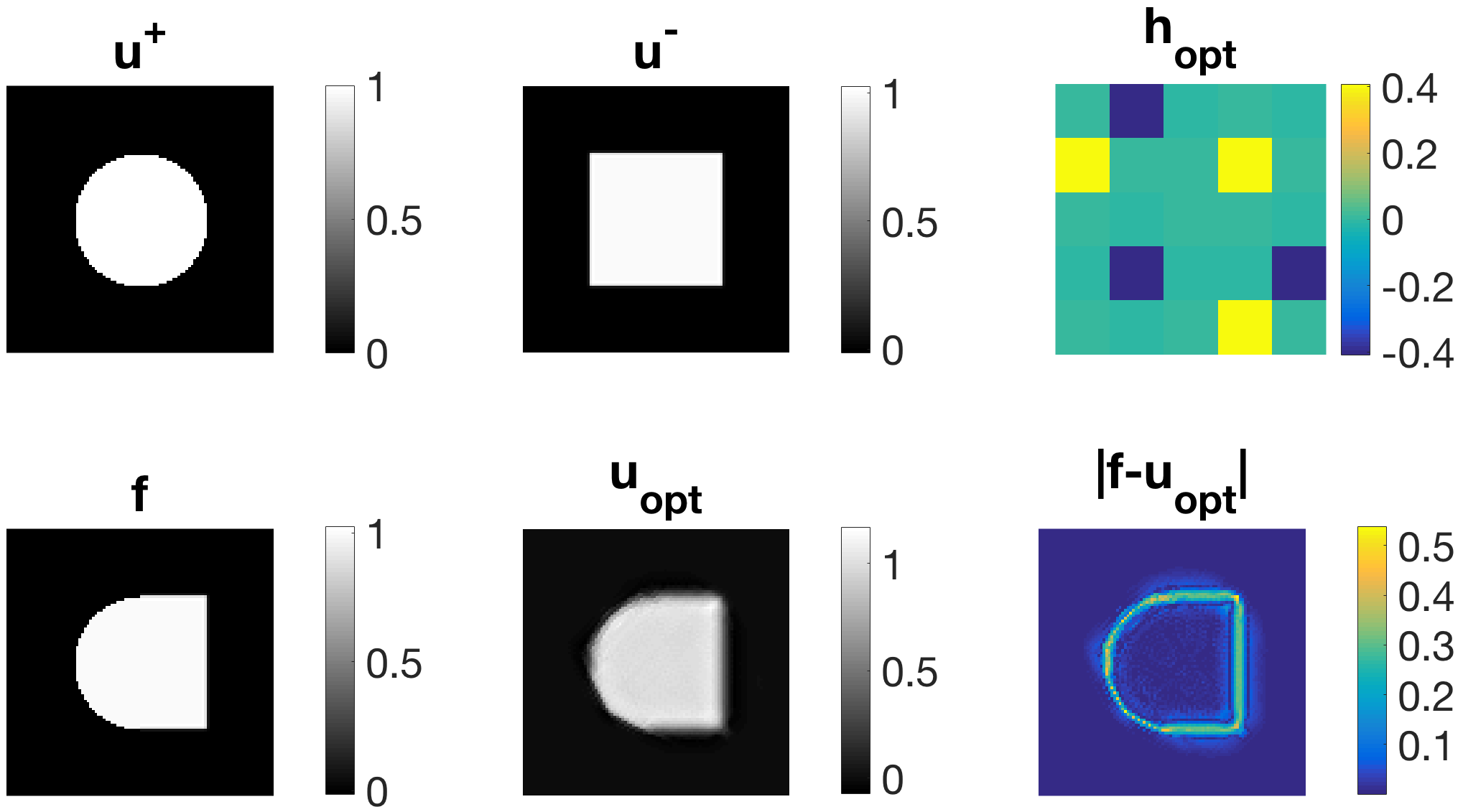}}\hfill
\subfigure[$J(u^+;\hat{h})=1$, $J(u^-;\hat{h})=63$]{\includegraphics[height=4cm]{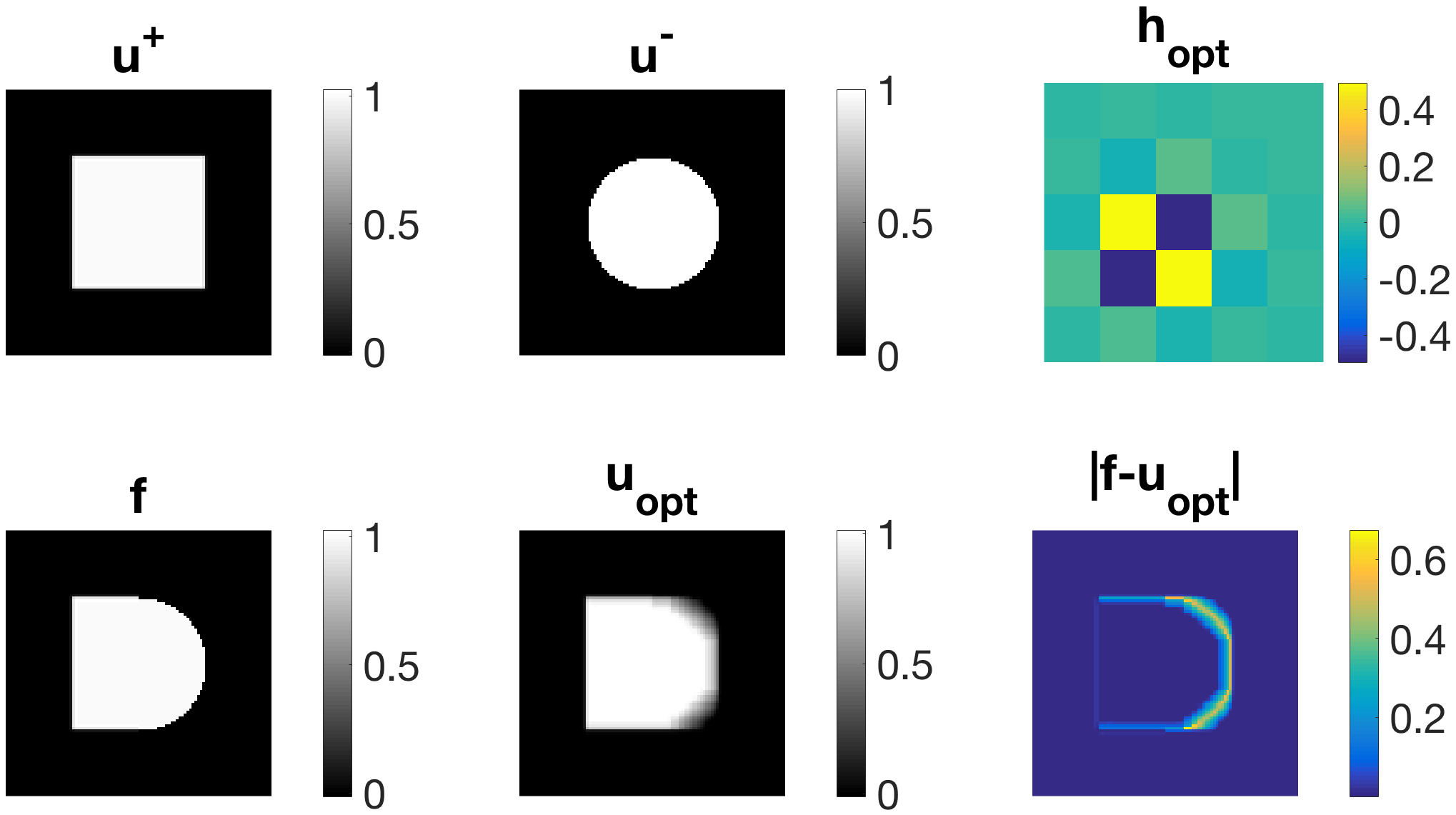}}
\caption{Our learning framework facilitates distinguishing shapes, angles and scales.}
\label{fig:shapes}
\end{figure}

Figure \ref{fig:shapes} shows a variety of experiments aiming at finding an optimal filter function $\hat{h}$, which favours the specific texture, orientation or scale in input image $u^+$ and disfavours the one present in $u^-$. Again, we perform reconstruction according to \eqref{eq:reconstruction} choosing $\eta = 1$ and $\sigma^2 = 0.005$, where the given image $f$ is a noise-free combination of $u^+$ and $u^-$.
We state the functional evaluations of $J(u^+;\hat{h})$ and $J(u^-;\hat{h})$ below the respective figures. It can be observed that the former are in most cases significantly smaller than the latter, confirming the usefulness of model \eqref{eq:basic} and Algorithm \eqref{eq:algorithm}.
In (a), we can clearly see that the left-hand-side can be almost perfectly reconstructed in $\hat{u}$ while undesired artefacts are occurring for the unfavourable horizontal stripes. Similar results can be observed in (b), where the optimal filter is sparser than the one in (a).
In (c), we have diagonal stripes in different angles as input images. Again, the left-hand side is almost perfectly reconstructed whereas the stripes on the right-hand side appear blurred. In (d), the diagonal stripes are only one pixel thick and hence the appearance of the filter changes significantly. The regulariser is able to reconstruct the left-hand side of $f$ very well.
We exchange $u^+$ and $u^-$ in (e) and (f). First, the circle is the desired and the square is the undesired input signal. The opposite case holds true for example (f). Here, again the diagonal-shaped filter performs best and blurs the circular structure enforcing edges in vertical and horizontal direction.

\begin{figure}
\centering
\subfigure[Increasing the filter size to $7 \times 7$ enables discovery of a filter that almost assumes a circular shape itself. $J(u^+;\hat{h})=123$, $J(u^-;\hat{h})=270$, $\eta = 0.1$, $\sigma^2 = 0.005$]{\includegraphics[height=4cm]{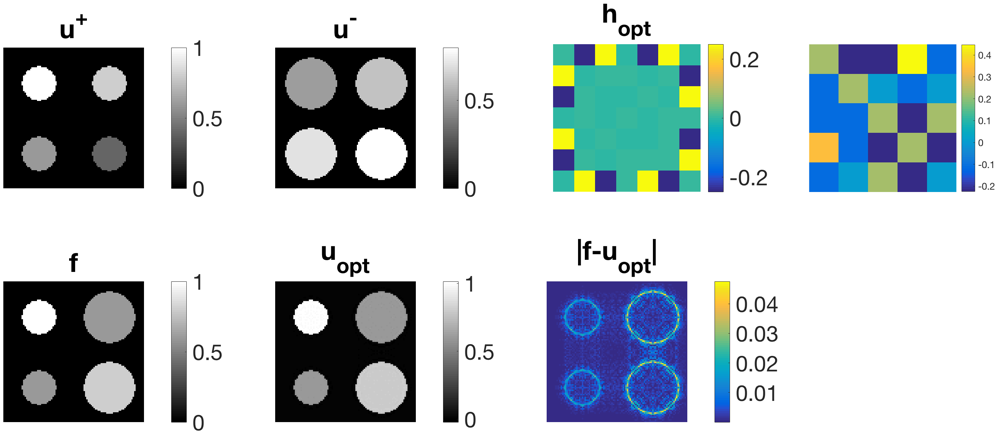}}\hfill\subfigure[Denoising experiment for multiple input images. $\eta = 1$, $\sigma^2 = 0.005$]{\includegraphics[height=4cm]{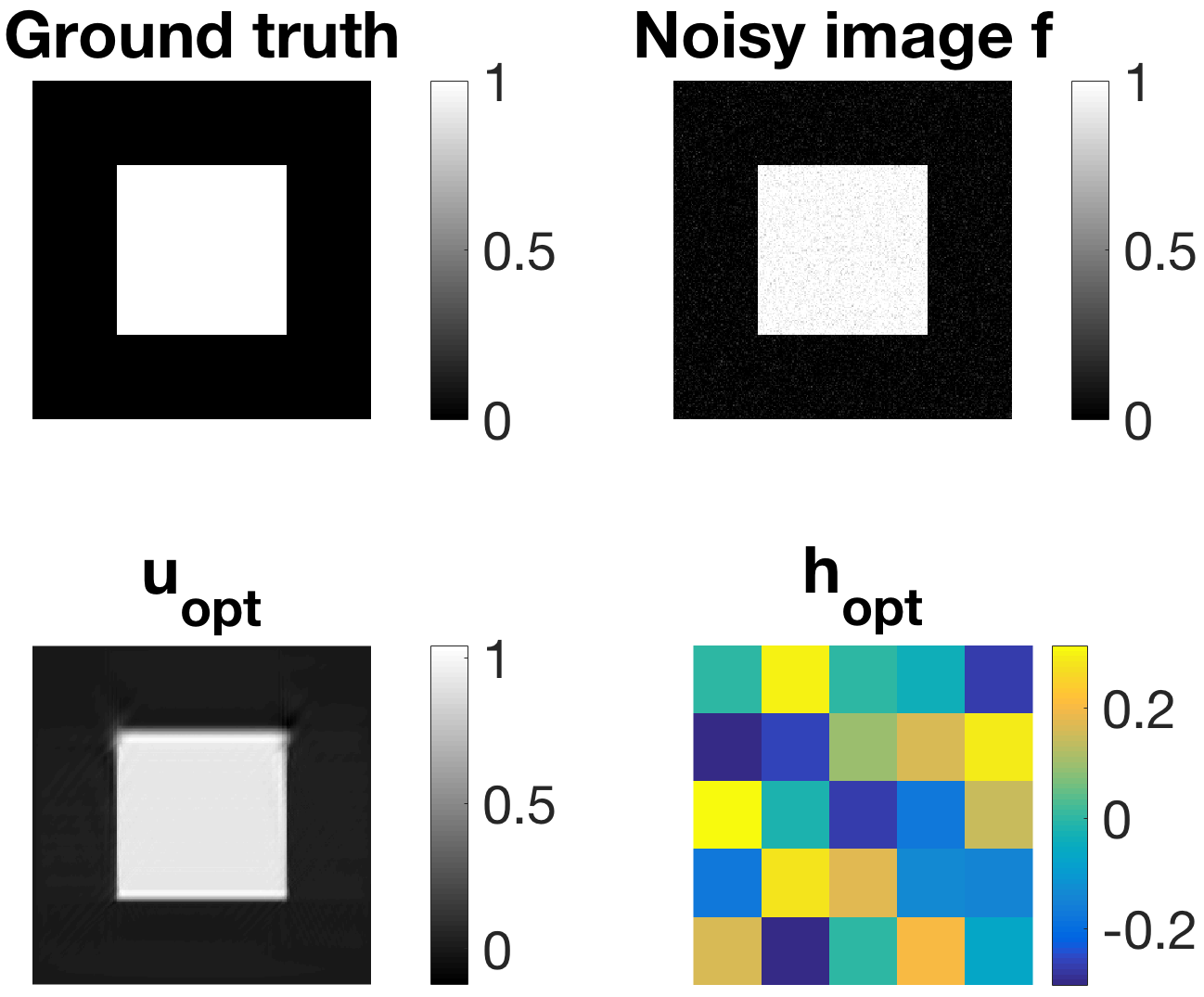}}
\caption{Further experiments: Increasing filter size and number of input images.}
\label{fig:experiments}
\end{figure}

\begin{figure}[h]
\centering
\includegraphics[height=2.5cm]{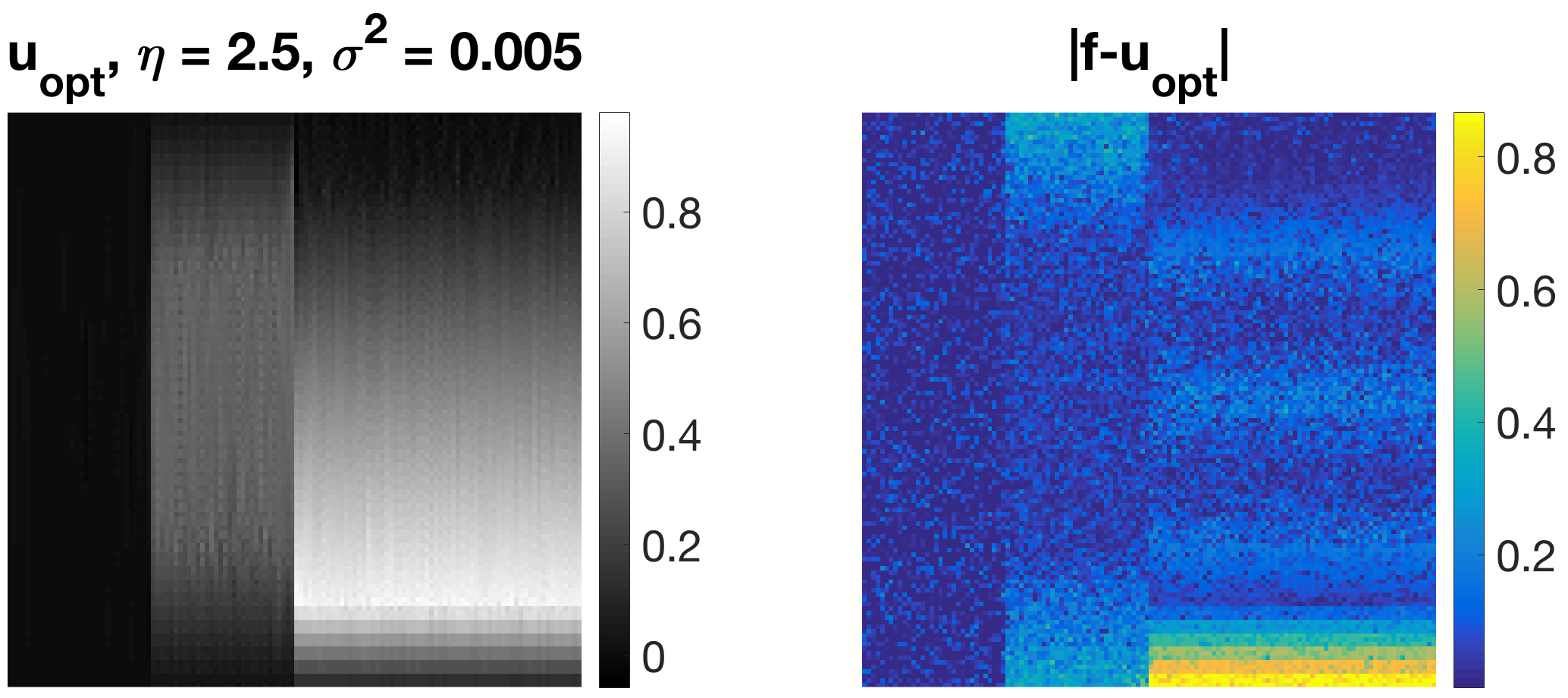}\hspace{1cm}\includegraphics[height=2.5cm]{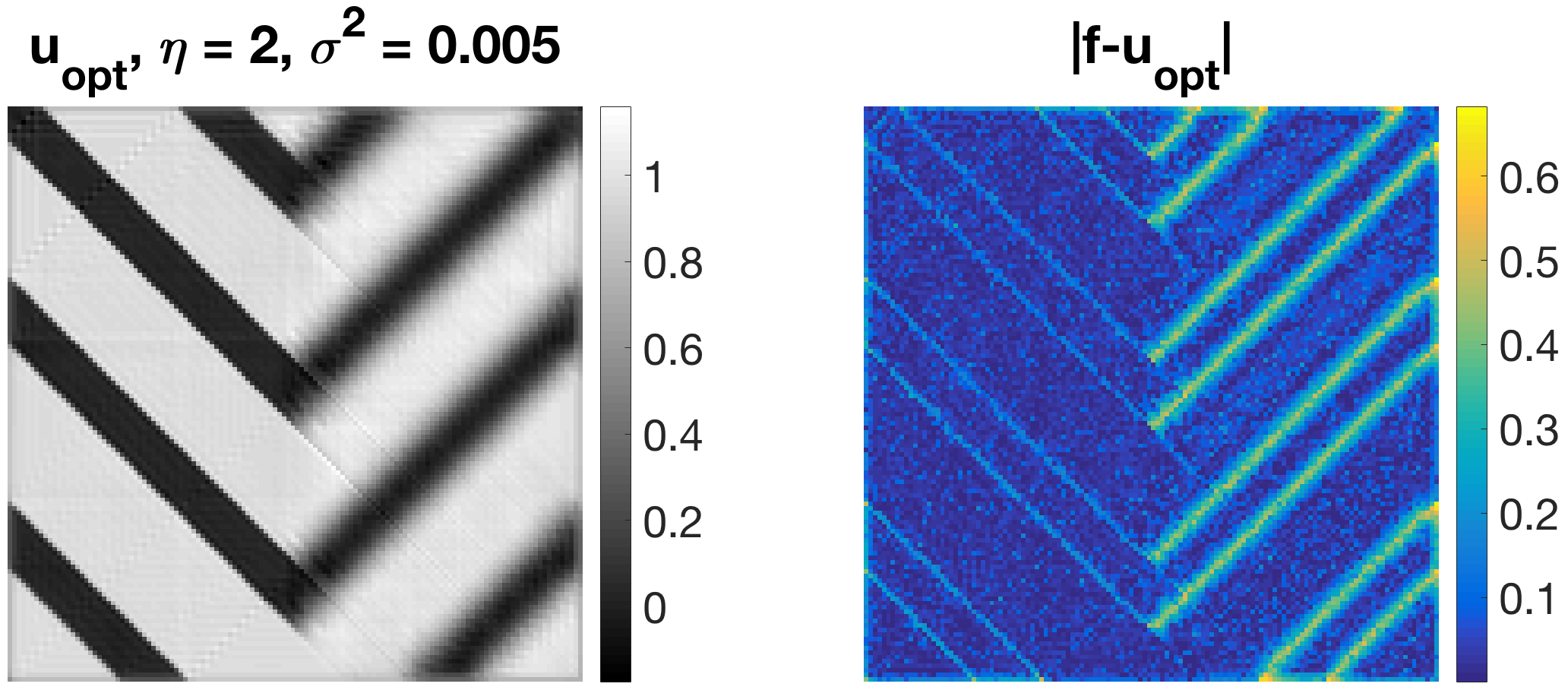}
\caption{Denoising performance of filters in Figure \ref{fig:shapes} (a) and (c).}
\label{fig:experiments2}
\end{figure}

Sometimes it is necessary to find suitable filters by increasing their size, as can be seen in Figure \ref{fig:experiments} (a). On the right-hand side we can see the optimal filter calculated assuming a size of $5 \times 5$. By increasing the size slightly to $7 \times 7$, we obtain a much more reasonable filter being able to recover the circles. The filter itself almost resembles a circle with a radius of three pixels.
In (b), given the ground truth image on the top left and four rotated versions of it in angles between 0 and 45 degrees as favourable input as well as five noise signals as negative input, we obtain the filter on the bottom right, which performs surprisingly well at denoising the image on the top right.

In Figure \ref{fig:experiments2}, convincing denoising results can be achieved for examples (a) and (c) in Figure \ref{fig:shapes}, adding noise to the ground truth image $f$ and using the calculated optimal filter. One can clearly see that the structures on the left-hand sides are denoised better.

We would like to remark that the setup with one filter is indeed rather simple and cannot mimic 2D differential-based filters like TV or alike, but therefore it is even more surprising that the filters presented above perform really well.

\section{Conclusions and Outlook}

Starting from the model in \cite{PAMM}, we derived a more generalised formulation suitable for minimisation with respect to multi-dimensional filter functions. In addition, our flexible framework allows for multiple desired and undesired input signals.

We were able to reproduce different common first- and second-order regularisers such as TV and TV$^2$ in the 1D case. Furthermore, we created a new family of non-derivative-based regularisers suitable for specific types of images. Also, we showed that specific shapes such as diagonal stripes can be eliminated while applying such parametrised regularisers. In addition, we believe that our learning approach is suitable for distinguishing between different shapes, scales and textures. A great advantage and novelty is that we are able to include both favourable and unfavourable input signals in our framework.

Regarding numerical implementation, we would like to stress that for computational simplicity in combination with the CVX framework we have only considered Dirichlet boundary conditions so far, but will use different boundary conditions (like the more suitable Neumann boundary conditions) in the future. 

We further assume that the ansatz of filter functions respectively convolutions as parametrisations for the regularisers is too generic especially for denoising tasks. It will be interesting to look into dictionary-based sparsity approaches, and to then learn basis functions with the presented quotient model.

Moreover, future work might include applications in biomedical imaging such as reconstruction in CT or MRI as well as denoising or object detection in light microscopy images. In \cite{Zeune}, the authors present a multiscale segmentation method for circulating tumour cells, where they are able to detect cells of different sizes. Using our model, we believe that shape or texture priors incorporated in sparsity-based regularisers could be well improved. One possible application could be mitotic cell detection (cf.\ \cite{MitosisAnalyser}).

\section{Acknowledgements}

MB acknowledges support from the Leverhulme Trust early career fellowship "Learning from mistakes: a supervised feedback-loop for imaging applications" and the Newton Trust. GG acknowledges support by the Israel Science Foundation (grant 718/15). JSG acknowledges support by the NIHR Cambridge Biomedical Research Centre. CBS acknowledges support from Leverhulme Trust project ’Breaking the non-convexity barrier’, EPSRC grant ’EP/M00483X/1’, EPSRC centre ’EP/N014588/1’, the Cantab Capital Institute for the Mathematics of Information, and from CHiPS (Horizon 2020 RISE project grant).

\subsubsection*{Data Statement.} The corresponding MATLAB\textsuperscript{\textregistered} code is publicly available on Apollo - University of Cambridge Repository (\url{https://doi.org/10.17863/CAM.8419}).

%\bibliographystyle{unsrt}
%\bibliography{arXiv}

\begin{thebibliography}{4}

\bibitem{KSVD} Aharon, M., M. Elad, and A. Bruckstein. "K-SVD: An Algorithm for Designing Overcomplete Dictionaries for Sparse Representation." IEEE Transactions on signal processing 54.11:4311-4322 (2006).

\bibitem{GroundStates}Benning, M., Burger, M.: Ground states and singular vectors of convex 
variational regularization methods. Methods and Applications of Analysis 20, no. 4, 295--334 (2013)

\bibitem{PAMM}Benning, M., Gilboa, G., Sch\"onlieb, C.-B.: Learning parametrised regularisation functions via quotient minimisation. PAMM 16.1, 933--936 (2016)

\bibitem{PALM}Bolte, J., Sabach, S., Teboulle, M.: Proximal alternating linearized minimization for nonconvex and nonsmooth problems. Mathematical Programming 146.1-2, 459--494 (2014)

\bibitem{BLUB}Bresson, X., Laurent, T., Uminsky, D., Brecht, J.V.: Convergence and energy landscape for Cheeger cut clustering. Advances in Neural Information Processing Systems (2012)

\bibitem{BKC}Brox, T., Kleinschmidt, O., Cremers, D.: Efficient nonlocal means for denoising of textural patterns. IEEE Transactions on Image Processing 17.7, 1083--1092 (2008)

\bibitem{SparseRev} Bruckstein, A.M., D.L. Donoho, and M. Elad. "From sparse solutions of systems of equations to sparse modeling of signals and images." SIAM review 51.1 (2009): 34-81.

\bibitem{LearningRS}De los Reyes, J.C., Sch\"onlieb, C.-B.: Image denoising: Learning the noise model via nonsmooth PDE-constrained optimization. Inverse Probl. Imaging 7.4, 1139--1155 (2013)

\bibitem{LearningRST}De los Reyes, J.C., Sch\"onlieb, C.-B., Valkonen, T.: Bilevel Parameter Learning for Higher-Order Total Variation Regularisation Models. Journal of Mathematical Imaging and Vision, 1--25 (2016)

\bibitem{G}Gilboa, G.: Expert Regularizers for Task Specific Processing. International Conference on Scale Space and Variational Methods in Computer Vision, Springer Berlin Heidelberg (2013)

\bibitem{MitosisAnalyser}Grah, J.S., Harrington, J., Koh, S.B., Pike, J., Schreiner, A., Burger, M., Sch\"onlieb, C.-B., Reichelt, S.: Mathematical Imaging Methods for Mitosis Analysis in Live-Cell Phase Contrast Microscopy. arXiv preprint arXiv:1609.04649 (2016)

\bibitem{CVX}Grant, M., Boyd, S.: CVX: Matlab software for disciplined convex programming, version 2.1, \url{http://cvxr.com/cvx} (2016)

\bibitem{HeinBuehler}Hein, M., B\"uhler, T.: An inverse power method for nonlinear eigenproblems with applications in 1-spectral clustering and sparse PCA. Advances in Neural Information Processing Systems (2010)

\bibitem{LearningKP}Kunisch, K., Pock, T.: A bilevel optimization approach for parameter learning in variational models. SIAM Journal on Imaging Sciences 6.2, 938--983 (2013)

\bibitem{kurdyka1998gradients}Kurdyka, K.: On gradients of functions definable in o-minimal structures. Annales de l'institut Fourier 48.3, 769--783 (1998) 

\bibitem{lojasiewicz1963}\L ojasiewicz, S.: Une propri\'{e}t\'{e} topologique des sous-ensembles analytiques r\'{e}els. Les \'{e}quations aux d\'{e}riv\'{e}es partielles, 87--89 (1963)

\bibitem{SparseCoding} V. Papyan, Y. Romano, and M. Elad. "Convolutional Neural Networks Analyzed via Convolutional Sparse Coding." arXiv preprint arXiv:1607.08194 (2016).

\bibitem{ISSD}Schmidt, M.F., Benning, M., Sch\"onlieb, C.-B.: Inverse Scale Space Decomposition. arXiv preprint arXiv:1612.09203 (2016)

\bibitem{Zeune}Zeune, L., van Dalum, G., Terstappen, L.W.M.M., van Gils, S.A., Brune, C.: Multiscale Segmentation via Bregman Distances and Nonlinear Spectral Analysis. arXiv preprint arXiv:1604.06665 (2016)

\end{thebibliography}

\end{document}